\documentclass[11pt]{article}
\usepackage{amsthm, amsmath, amssymb, amsfonts, url, booktabs, tikz, setspace, fancyhdr, bm}
\usepackage{amsthm}

\usepackage[hidelinks]{hyperref}
\usepackage{geometry}
\usepackage{hyperref, enumerate}
\usepackage[shortlabels]{enumitem}
\usepackage[babel]{microtype}
\usepackage[english]{babel}
\usepackage[capitalise]{cleveref}
\usepackage{comment}
\usepackage{bbm}
\usepackage{csquotes}
\usepackage{diagbox}
\usepackage{mathabx}
\usepackage{tikz}
\allowdisplaybreaks
\usepackage{float}

\counterwithin{figure}{section}

\newtheorem{theorem}{Theorem}[section]
\newtheorem{prop}[theorem]{Proposition}

\newtheorem{cor}[theorem]{Corollary}

\newtheorem{claim}[theorem]{Claim}

\theoremstyle{definition}

\newtheorem{exam}[theorem]{Example}

\newtheorem{rmk}[theorem]{Remark}

\newlist{Case}{enumerate}{2}
\setlist[Case, 1]{%
    label           =   {\bfseries Case \arabic*.},
    labelindent=1em ,labelwidth=1.3cm, labelsep*=1em, leftmargin =!
}
\setlist[Case, 2]{%
    label           =   {\bfseries Subcase \arabic{Casei}.\arabic*.},
    labelindent=-1em ,labelwidth=1.3cm, labelsep*=1em, leftmargin =!
}

\newenvironment{poc}{\begin{proof}[Proof of claim]}{\end{proof}}

\usepackage{todonotes}

\title{Exact values and improved bounds on $k$-neighborly families of boxes}
\author{
Xinbu Cheng\thanks{Laboratory of Mathematics and Complex Systems, Ministry of Education, School of Mathematical Sciences, Beijing Normal University, Beijing, China. Emails: chengxinbu2006@sina.com and W15903853878@163.com.}
\and
Meiqin Wang\footnotemark[1]
\and
Zixiang Xu\thanks{Extremal Combinatorics and Probability Group (ECOPRO), Institute for Basic Science (IBS), Daejeon, South Korea. Email: zixiangxu@ibs.re.kr. Supported by IBS-R029-C4.}
\and
Chi Hoi Yip\thanks{Department of Mathematics, University of British Columbia, Vancouver, Canada. Email: kyleyip@math.ubc.ca.}\\
}
\date{}

\begin{document}

\maketitle
\begin{abstract}
  A finite family $\mathcal{F}$ of $d$-dimensional convex polytopes is called $k$-neighborly if $d-k\le\textup{dim}(C\cap C')\le d-1$ for any two distinct members $C,C'\in\mathcal{F}$. In 1997, Alon initiated the study of the general function $n(k,d)$, which is defined to be the maximum size of $k$-neighborly families of standard boxes in $\mathbb{R}^{d}$. Based on a weighted count of vectors in $\{0,1\}^{d}$, we improve a recent upper bound on $n(k,d)$ by Alon, Grytczuk, Kisielewicz, and Przes\l awski for any positive integers $d$ and $k$ with $d\ge k+2$. In particular, when $d$ is sufficiently large and $k\ge 0.123d$, our upper bound on $n(k,d)$ improves the bound $\sum_{i=1}^{k}2^{i-1}\binom{d}{i}+1$ shown by Huang and Sudakov exponentially. 
  
  Furthermore, we determine that $n(2,4)=9$, $n(3,5)=18$, $n(3,6)=27$, $n(4,6)=37$, $n(5,7)=74$, and $n(6,8)=150$. The stability result of Kleitman's isodiametric inequality plays an important role in the proofs.
\end{abstract}

\section{Introduction}\label{sec:Introduction}
A standard box in $\mathbb{R}^{d}$ is a set of the form $K=K_{1}\times \cdots \times K_{d}$, where $K_{i}\subseteq \mathbb{R}$ is a closed interval for each $i\in [d]$. A finite family $\mathcal{C}$ of $d$-dimensional standard boxes is called $k$-neighborly if the dimension of the intersection of any two boxes is in $\{d-k,d-k+1,\ldots,d-1\}$. For positive integers $1\le k\le d$, the function $n(k,d)$ is defined to be the maximum number of elements in a $k$-neighborly family of $d$-dimensional standard boxes. The $1$-neighborly families of standard boxes were widely studied several decades ago. For example, Zaks~\cite{1987Zaks} showed the maximum size of a $1$-neighborly family is closely related to the famous Graham-Pollak theorem~\cite{1972GrahamPollak}, which gives that $n(1,d)=d+1$. There have been several other beautiful proofs of Graham-Pollak theorem~\cite{1982JGT,2008JCTA, 2013DMCountingProof}. In 1997, Alon~\cite{1997Alon} generalized the problem to $k$-neighborly families and proved the equivalence between the maximum size of $k$-neighborly families of standard boxes and a variant of bipartite coverings. More precisely, he showed that for $1\le k\le d$, $n(k,d)$ is precisely the maximum number of vertices of a complete graph that admits a bipartite covering of order $k$ and size $d$.

Based on an elegant polynomial method and an explicit construction, Alon~\cite{1997Alon} proved the following theoretical bounds on $n(k,d)$.

\begin{theorem}[\cite{1997Alon}]\label{thm:Alon1997}
    For $1\leqslant k\leqslant d$, we have
   \begin{equation*}
     \prod\limits_{i=0}^{k-1}\bigg(\left\lfloor\frac{d+i}{k}\right\rfloor+1\bigg) \le n(k,d)\leqslant \sum_{i=0}^{k}2^{i}\binom{d}{i}.
   \end{equation*}
\end{theorem}

About one decade ago, Huang and Sudakov~\cite{huang2012counterexample} improved the above upper bound via a clever rank argument and a trick of Peck~\cite{1984Peck}.
\begin{theorem}[\cite{huang2012counterexample}]\label{thm:Huang2012}
    For $1\leqslant k\leqslant d$, we have
   \begin{equation*}
       n(k,d)\leqslant 1+\sum_{i=1}^{k}2^{i-1}\binom{d}{i}.
   \end{equation*}
\end{theorem}

 Very recently, Alon, Grytczuk, Kisielewicz, and Przesławski~\cite{2022neighbor} provided an upper bound for $n(k,d)$ of a different type as follows.

  \begin{theorem}[{\cite[Theorem 1]{2022neighbor}}]\label{thm:previous}
For every $1 \leqslant k \leqslant k+2 i-2 \leqslant d-1$, where $i$ is an arbitrary positive integer,
$$
n(k, d) \leqslant 2^{d-i}+\sum_{j=0}^{\lceil(k+2i-2) / 2\rceil}\binom{d}{j}.
$$
\end{theorem}

They also provided a general construction of $k$-neighborly families of boxes, which yields a new lower bound on $n(k,d)$. In particular, if $k$ is fixed and $d \to \infty$, they \cite[Theorem 2]{2022neighbor} showed that
 $n(k,d)\ge (1-o(1))\frac{d^{k}}{k!}, $ 
which improves Theorem~\ref{thm:Alon1997}. They also proved a nice recursive lower bound on $n(k,d)$ \cite[Proposition 2]{2022neighbor}. One novelty of their paper is to estimate $n(k,d)$ from a perspective related to the Hamming metric; see Section~\ref{sec:prelim}. 

Inspired by their work \cite{2022neighbor}, in this paper, we establish an improved upper bound on $n(k,d)$ for any positive integers $k$ and $d$ with $k\le d-2$. In particular, we prove that $n(d-1,d)=3\cdot 2^{d-2}$ as a special case; see Corollary~\ref{cor:d-1} (see also \cite[Proposition 1]{2022neighbor}). Our new upper bounds are closely related to an important configuration, namely $B_{k}^{(d)}$, in coding theory. More precisely, when $k=2t$, $B_{k}^{(d)}$ is a $d$-dimensional Hamming ball of radius $t$, that is, a collection of binary vectors of length $d$ having at most $t$ $1$-coordinates; when $k=2t+1$, $B_{k}^{(d)}$ is the Cartesian product of $\{0,1\}$ and a $(d-1)$-dimensional Hamming ball of radius $t$. It is clear that \begin{equation}\label{Bkd}
    \big|B_{k}^{(d)}\big|= 
    \begin{cases}
        \sum\limits_{i=0}^{t}\binom{d}{i} &  k=2t; \\ 
   2\sum\limits_{i=0}^{t}\binom{d-1}{i} &  k=2t+1.
    \end{cases}
\end{equation}  

   \begin{theorem}\label{thm:main}
    Let $1\leqslant k\leqslant d-1$ be integers. If $d-k$ is even, then
\begin{equation*}
     n(k,d) \leq \sum\limits_{j=0}^{\frac{d-k-2}{2}}\bigg(\frac{1}{2^{j+1}}-\frac{1}{2^{d-k-j}}\bigg) \big|B_{k+2j}^{(d)}\big|+2^{\frac{d+k}{2}}.
\end{equation*}
If $d-k$ is odd, then
    \begin{equation*}
     n(k,d) \leq    \sum\limits_{j=0}^{\frac{d-k-3}{2}}\bigg(\frac{1}{2^{j+1}}-\frac{1}{2^{d-k-j}}\bigg) \big|B_{k+2j}^{(d)}\big|+\frac{1}{2^{\frac{d-k+1}{2}}}\big|B_{d-1}^{(d)}\big|+2^{\frac{d+k-1}{2}}.
    \end{equation*}   
\end{theorem}

It is not hard to show that Theorem~\ref{thm:main} always improves Theorem~\ref{thm:previous}. Moreover, when $k,d$ are sufficiently large and $k\ge0.123d$, our upper bounds are asymptotically smaller than the upper bounds in Theorem~\ref{thm:Huang2012}. We refer to the detailed computations for such a comparison in Section~\ref{sec:comparison}.

When $k$ is fixed, and $d$ is sufficiently large, the following theorem further improves the upper bound on $n(k,d)$; see Remark~\ref{d_sufficiently_large}.
\begin{theorem}\label{thm:main2}
    Let $1\leqslant k\leqslant d-1$ be integers. 
    If $d-k$ is even, then
\begin{equation*}
     n(k,d) \leq \max \left\{\sum_{j=0}^k \binom{d}{j},  \sum\limits_{j=1}^{\frac{d-k-2}{2}}\bigg(\frac{1}{2^{j+1}}-\frac{1}{2^{d-k-j}}\bigg) \big|B_{k+2j}^{(d)}\big|+2^{\frac{d+k}{2}}\right\}.
\end{equation*}
If $d-k$ is odd, then
    \begin{equation*}
     n(k,d) \leq    \max \left\{\sum_{j=0}^k \binom{d}{j}, \sum\limits_{j=1}^{\frac{d-k-3}{2}}\bigg(\frac{1}{2^{j+1}}-\frac{1}{2^{d-k-j}}\bigg) \big|B_{k+2j}^{(d)}\big|+\frac{1}{2^{\frac{d-k+1}{2}}}\big|B_{d-1}^{(d)}\big|+2^{\frac{d+k-1}{2}}\right\}.
    \end{equation*}   
\end{theorem}

Alon~\cite{1997Alon} remarked that even the precise determination of $n(2,d)$ seems difficult. As a consequence of Theorem~\ref{thm:main2}, we improve the best-known upper bound on $n(2,d)$ for any $d\le 9$. Note that Theorem~\ref{thm:main2} only implies that $n(2,10)\leq 108$, which is worse than $n(2,10) \leq 101$ from Theorem~\ref{thm:Huang2012}. Nevertheless, we further develop the ideas in the proof of Theorem~\ref{thm:main2} in Theorem~\ref{thm:main3}, which in particular gives a better upper bound $n(2,10) \leq 95$.
\begin{cor}
\begin{equation*}
  n(2,4)=9;\ n(2,5)\le 14;\  n(2,6)\le 21;\ n(2,7)\le 29;\ n(2,8)\le 45;\ n(2,9)\le 70;\ n(2,10) \le 95.
\end{equation*}
   
\end{cor}

Another goal of our paper is to determine exact values on $n(k,d)$ when $k$ and $d$ are small. In this direction, the authors in~\cite{2022neighbor} took advantage of the Gurobi solver to compute lower bounds on $n(k,d)$ for small $k,d$. Combining our upper bounds in Theorem~\ref{thm:main}, we can immediately obtain the following corollary.
    \begin{cor}\label{cor:ToMain}
        \begin{equation*}
       n(3,5)=18;\ n(4,6)=37;\ 27\le n(3,6)\le 28;\ 74\le n(5,7)\le 75;\ 150\le n(6,8)\le 151.
        \end{equation*}
    \end{cor}

Using a stability version of Kleitman's theorem (see Theorem~\ref{thm:StabilityFrankl}), we further determine the exact values of $n(3,6)$, $n(5,7)$ and $n(6,8)$.
\begin{theorem}\label{thm:ShortTight}
    \begin{equation*}
      n(3,6)=27;\  n(5,7)=74;\ n(6,8)=150.
    \end{equation*}
\end{theorem}

We collect some exact values and bounds on $n(k,d)$ for small $k$ and $d$ in Table~\ref{tab:values}. The first two are lower bounds and the others are upper bounds. We also list all the improved bounds on $n(k,d)$ for $k,d\le 20$ and $d-k\ge 2$ in Appendix.

\begin{table}[H]
    \centering
\begin{tabular}{|c|c|c|c|c|c|c|}
	\hline
$d$	& $k$ & Thm~\ref{thm:Alon1997}~\cite{1997Alon} & Gurobi~\cite[Table 2]{2022neighbor}  & Our new upper bounds  &Thm 1.2~\cite{huang2012counterexample}  &Thm 1.3~\cite{2022neighbor}\\
	\hline
5	& 3 & 18 & 18 &\textbf{18}  & 66 & 32 \\
	\hline
6	& 3 &27  &27  & \textbf{27} & 117 &  54\\
	\hline
6	& 4 & 36 & 37 & \textbf{37} & 237 & 54 \\
	\hline
7	& 5 & 72 & 74 & \textbf{74} & 806 & 128 \\
	\hline
8	& 6 & 144 & 150 & \textbf{150} & 2641 & 221 \\
\hline
	\hline
4	& 2 & 9 & -- &\textbf{9}  & 17 & 13 \\
\hline
5	& 2 &  12&  12& 14 & 26 & 22 \\
	\hline
6	& 2 &  16&  16& 21  & 37 & 38 \\
	\hline
 7	& 2 &  20&  21& 29 & 50 & 61\\
	\hline
8	& 2 & 25& 27 & 45 & 65 & 101 \\
 \hline
 9	& 2 & 30 & -- & 70 & 82 & 174 \\
 \hline
 10	& 2 & 36 & -- & 95 & 101 & 304 \\
 \hline
\end{tabular}
\caption{Some exact values and bounds of $n(k,d)$ (the bold numbers are exact values)}
    \label{tab:values}

\end{table}

\medskip
{\bf \noindent Notations.}
We will use $\binom{[d]}{k}$ to denote the set of all vectors with length $d$ and $k$ $1$-coordinates and $\binom{[d]}{\leq k}$ to denote the set of all vectors with length $d$ and at most $k$ $1$-coordinates, that is, the $d$-dimensional Hamming ball of radius $k$ centered at $\boldsymbol{0}$.
We use $X\square Y:=\{(\boldsymbol{x},\boldsymbol{y}):\boldsymbol{x}\in X, \boldsymbol{y}\in Y\}$ to denote the Cartesian product of two sets $X$ and $Y$. For example, $\{0,1\}\square\binom{[d-1]}{\le k}$ can be viewed as the Cartesian product of $\{0,1\}$ and a $(d-1)$-dimensional Hamming ball of radius $k$. For two vectors $\boldsymbol{v}=(v_{1},v_{2},\ldots,v_{d})$ and $\boldsymbol{u}=(u_{1},u_{2},\ldots,u_{d})$ in $\{0,1\}^{d}$, denote $\boldsymbol{w}=\boldsymbol{v}\lor \boldsymbol{u}=(w_{1},w_{2},\ldots,w_{d})$, where for each $1\le i\le d$, $w_{i}=0$ if both of $u_{i}$ and $v_{i}$ are equal to $0$, and $w_{i}=1$ otherwise. 

\medskip
\noindent\textbf{Structure of the paper.}
The rest of this paper is organized as follows. We further introduce the extremal problem we focus on and collect some useful tools in Section~\ref{sec:prelim}. We present the proofs of Theorems~\ref{thm:main} and Theorem~\ref{thm:main2} in Section~\ref{sec:General}. The proofs of exact values of $n(k,d)$ are in Section~\ref{sec:values}. Finally, we conclude and discuss some related work in Section~\ref{sec:Conclusion}. We also list the improved bounds on $n(k,d)$ for relatively small $k$ and $d$ in Appendix.

\section{Preliminaries}\label{sec:prelim}

Very recently, Alon, Grytczuk, Kisielewicz, and Przesławski~\cite{2022neighbor} provided a new interpretation of the function $n(k,d)$. They showed that the problem of determining the function $n(k,d)$ is equivalent to an extremal problem on a set of vectors with restricted Hamming distances. 

Recall that for vectors $\boldsymbol{v}_{1},\boldsymbol{v}_{2}\in\{0,1\}^{d}$, the \emph{Hamming distance} between $\boldsymbol{v}_{1}$ and $\boldsymbol{v}_{2}$ is the number of coordinates they differ, denoted by $d_{H}(\boldsymbol{v}_{1},\boldsymbol{v}_{2})$. Following the notions in \cite{2022neighbor}, we call $*$ a \emph{joker}. We can define the Hamming distance in $\{0,1,*\}^d$ similarly. For two vectors $\boldsymbol{u}_{1},\boldsymbol{u}_{2}\in\{0,1,*\}^{d}$, the Hamming distance between them is defined to be the number of coordinates they differ in non-joker positions. For example, the Hamming distance between $(1,1,*,*,0,0)$ and $(*,*,1,0,*,1)$ is exactly $1$. This generalization of Hamming distance is closely related to the addressing problem for graphs \cite{1972GrahamPollak}.

The following proposition describes an explicit connection between $n(k,d)$ and the size of an extremal set vectors in $\{0,1,*\}^d$ with restricted Hamming distance:

\begin{prop}[{\cite[Section 2]{2022neighbor}}]\label{prop:Equivalence}
Let $k,d\in\mathbb{N}$. The function $n(k,d)$ is equal to the maximum size of $U \subseteq \{0,1,*\}^{d}$ such that the Hamming distance between $\boldsymbol{u}_{i}$ and $\boldsymbol{u}_{j}$ belongs to $\{1,2,\ldots,k\}$ for any distinct $\boldsymbol{u}_{i},\boldsymbol{u}_{j}\in U$. 
\end{prop}

The starting point of the proof of our main results is the above interpretation of $n(k,d)$.  Our proof of Theorem~\ref{thm:main} is based on a weighted count of vectors in $\{0,1\}^{d}$. As a preparation, we revisit a classical extremal result related to restricted Hamming distances, namely Kleitman's theorem~\cite{1966Kleitman}.

\begin{theorem}[\cite{1966Kleitman}]\label{thm:Kleitman}
    For integers $d>k$ and a set $A$ of vectors in $\{0,1\}^{d}$ with $d_{H}(\boldsymbol{u},\boldsymbol{v})\le k$ for any vectors $\boldsymbol{u},\boldsymbol{v}\in A$, then
\begin{equation}\label{equ:Kleitman}
    |A|\leq 
    \begin{cases}
        \sum\limits_{i=0}^{t}\binom{d}{i} &  k=2t; \\ 
   2\sum\limits_{i=0}^{t}\binom{d-1}{i} &  k=2t+1.
    \end{cases}
\end{equation}  
Moreover, the above upper bound can be attained by the configuration $B_{k}^{(d)}$.
\end{theorem}

There is a stability result for Kleitman's theorem proved by Frankl~\cite{2017CPCFrankl} which states that if a set $A$ of vectors in $\{0,1\}^{d}$, with pairwise Hamming distances bounded by $k$, is not contained in any copy of the configuration $B_{k}^{(d)}$, then the size of the set $A$ is much smaller than the optimal value $|B_{k}^{(d)}|$. This result can be also proved via a robust linear algebra method; see the recent work of Gao, Liu, and Xu~\cite{2022Stability}.

\begin{theorem}[\cite{2017CPCFrankl}]\label{thm:StabilityFrankl}
     For integers $d\ge k+2$ and a set $A$ of vectors in $\{0,1\}^{d}$ with $d_{H}(\boldsymbol{u},\boldsymbol{v})\le k$ for any vectors $\boldsymbol{u},\boldsymbol{v}\in A$,
     \begin{enumerate}
         \item If $k=2t$ and $A$ is not contained in any $d$-dimensional Hamming ball of radius $t$, then 
         \begin{equation}\label{equ:EvenStability}
             |A|\leqslant \sum_{i=0}^{t}\binom{d}{i}-\binom{d-t-1}{t}+1.
         \end{equation}
         \item If $k=2t+1$ and $A$ is not contained in any Cartesian product of $\{0,1\}$ and an $(d-1)$-dimensional Hamming ball of radius $t$, then
         \begin{equation}\label{equ:OddStability}
             |A| \le 2\sum_{i=0}^{t}\binom{d-1}{i} - \binom{d-t-2}{t}+1.
         \end{equation}
     \end{enumerate}
\end{theorem}

\section{Improved upper bounds on $n(k,d)$}\label{sec:General}

Throughout the section, let $U=\{\boldsymbol{u}_{1},\boldsymbol{u}_{2},\ldots,\boldsymbol{u}_{n}\}$ be a set of vectors in $\{0,1,*\}^{d}$ such that the Hamming distance between $\boldsymbol{u}_{i}$ and $\boldsymbol{u}_{j}$ belongs to $\{1,2,\ldots,k\}$ for any distinct $\boldsymbol{u}_{i},\boldsymbol{u}_{j}\in U$. In view of Proposition~\ref{prop:Equivalence}, $n(k,d)$ is the maximum size of such $U$, thus it suffices to study the upper bound on $|U|$.

For a vector $\boldsymbol{v}\in\{0,1\}^{d}$, we say $\boldsymbol{v}$ is \emph{covered by} a vector $\boldsymbol{u}\in U$ (equivalently, $\boldsymbol{u}$ \emph{covers} $\boldsymbol{v}$) if the coordinates of $\boldsymbol{v}$ and $\boldsymbol{u}$ only differ in the joker-positions of $\boldsymbol{u}$. 

\begin{claim}\label{3.1}
    For each vector $\boldsymbol{v}\in\{0,1\}^{d}$, there is at most one vector in $U$ covering $\boldsymbol{v}$.  
  \end{claim} 
  \begin{poc}
    Suppose there are two distinct $\boldsymbol{u}$ and $\boldsymbol{u'}$ in $U$ covering a vector $\boldsymbol{v}\in\{0,1\}^{d}$, by the definition of $U$, there is at least one coordinate of $\boldsymbol{u}$ and $\boldsymbol{u'}$ in non-joker position being different. However, as both of $\boldsymbol{u}$ and $\boldsymbol{u'}$ cover the vector $\boldsymbol{v}$, all of the coordinates of $\boldsymbol{u}$ and $\boldsymbol{u'}$ in non-joker positions should be the same, a contradiction.  
  \end{poc}

For a vector $\boldsymbol{v}=(v_{1},v_{2},\ldots,v_{d})\in\{0,1\}^{d}$, let $\overline{\boldsymbol{v}}:=\boldsymbol{1}-\boldsymbol{v}=(1-v_{1},1-v_{2},\ldots,1-v_{d})$. Let
\begin{equation*}
    V^{(t)}:=\{\boldsymbol{v}\in\{0,1\}^{d}:\ \boldsymbol{v}\ \text{is covered by some\ } \boldsymbol{u}\in U\ \text{with exactly $t$ jokers} \},
\end{equation*}
and
\begin{equation*}
    \overline{V^{(t)}}:=\{\overline{\boldsymbol{v}}\in\{0,1\}^{d}:\ \boldsymbol{v}\ \text{is covered by some\ } \boldsymbol{u}\in U\ \text{with exactly $t$ jokers} \}.
\end{equation*}
It is obvious that $|V^{(t)}|=| \overline{V^{(t)}}|$. The following claim is useful.

\begin{claim} \label{claim:disjoint}
For any $0\le i<\frac{d-k}{2}$, $V^{(0)},\overline{V^{(0)}}, V^{(1)},\overline{V^{(1)}},\ldots, V^{(i)},\overline{V^{(i)}}$ are pairwise disjoint.
\end{claim}
\begin{poc}
Suppose for some 
integers $0\le a,b\le i$, $V^{(a)}\cap\overline{V^{(b)}}\neq\emptyset$, then there is some vector $\boldsymbol{v}\in V^{(a)}\cap\overline{V^{(b)}}$. Since $\boldsymbol{v}\in V^{(a)}$, $\boldsymbol{v}$ is covered by some vector $\boldsymbol{u}_{1}\in U$ with $a$ jokers. Moreover, $\overline{\boldsymbol{v}}$ is covered by some vector $\boldsymbol{u}_{2}\in U$ with $b$ jokers. As $d_{H}(\boldsymbol{v},\overline{\boldsymbol{v}})=d$, we have $d_{H}(\boldsymbol{u}_{1},\boldsymbol{u}_{2})\ge d-a-b\ge d-2i>k$ as $i<\frac{d-k}{2}$, which contradicts to the assumption that $d_{H}(\boldsymbol{u}_{1},\boldsymbol{u}_{2})\le k$. Moreover, by Claim~\ref{3.1}, any vector $\boldsymbol{v}\in \{0,1\}^{d}$ cannot be in both of $V^{(a)}$ and $V^{(b)}$ ($\overline{V^{(a)}}$ and $\overline{V^{(b)}}$, respectively) for distinct $0\le a,b\le i$.
 The proof is finished.
\end{poc}  

The following claim (also independently observed in~\cite{2022neighbor}) is also crucial in our proof. We provide a simple proof for the sake of completeness.

\begin{claim}\label{claim:UpperBoundSum}
  For any $0\le i<\frac{d-k}{2}$ and any two distinct vectors $\boldsymbol{a},\boldsymbol{b}\in\bigcup_{s=0}^{i}V^{(s)}$, $d_{H}(\boldsymbol{a},\boldsymbol{b})\le k+2i$. As a consequence, we have
  \begin{equation}
      \sum_{s=0}^{i}|V^{(s)}|\le |B_{k+2i}^{(d)}|.
  \end{equation}
\end{claim}
\begin{poc}
    As $\boldsymbol{a}$ and $\boldsymbol{b}$ are covered by some vectors $\boldsymbol{u}_{1}$ and $\boldsymbol{u}_{2}$ in $U$ with at most $i$ jokers, and $d_{H}(\boldsymbol{u}_{1},\boldsymbol{u}_{2})\le k$, we have $d_{H}(\boldsymbol{a},\boldsymbol{b})\le k+2i$. The claim follows from Theorem~\ref{thm:Kleitman}.
\end{poc}
\subsection{Proof of Theorem~\ref{thm:main}}
  
Equip each vector $\boldsymbol{v}\in\{0,1\}^{d}$ with a weight function $f(\boldsymbol{v})$. If $\boldsymbol{v}$ is covered by some vector $\boldsymbol{u}\in U$ (note that such a vector $\boldsymbol{u}$ is unique by Claim~\ref{3.1}) and $\boldsymbol{u}$ has exactly $t$ jokers, then we define $f(\boldsymbol{v})=\frac{1}{2^{t}}$. Otherwise, we define $f(\boldsymbol{v})=0$.

By definition of the function $f(\boldsymbol{v})$, we have
\begin{equation*}
    \sum_{\boldsymbol{v}\in\{0,1\}^{d}}f(\boldsymbol{v})=|U|.
\end{equation*}
 The above formula holds because each vector $\boldsymbol{u}\in U$ with $t$ jokers covers exactly $2^{t}$ vectors in $\{0,1\}^{d}$, and every binary vector $\boldsymbol{v}$ can be covered by at most one vector in $U$, which is shown in~Claim~\ref{3.1}. Therefore, it suffices to analyze the values of the function $f(\boldsymbol{v})$ for every $\boldsymbol{v}\in\{0,1\}^{d}$. 

 \begin{claim}\label{claim:overline}
For a vector $\boldsymbol{v}\in\overline{V^{(i)}}$ with $0\le i<\frac{d-k}{2}$, we have $f(\boldsymbol{v})\le\frac{1}{2^{d-k-i}}$.     
 \end{claim}
 \begin{poc}
 If $\boldsymbol{v}\in\overline{V^{(i)}}$ is not covered by any vector in $U$, then $f(\boldsymbol{v})=0$. Otherwise $\boldsymbol{v}$ is covered by some vector $\boldsymbol{u}\in U$ with $j$ jokers. And $\overline{\boldsymbol{v}}$ is covered by some vector $\boldsymbol{w}\in U$ with $i$ jokers, as $d_{H}(\boldsymbol{v},\overline{\boldsymbol{v}})=d$. So we have $d-j-i\le d_{H}(\boldsymbol{u},\boldsymbol{w})\le k$, which gives $j\ge d-k-i$. Therefore  $f(\boldsymbol{v})\le\frac{1}{2^{d-k-i}}$.    
 \end{poc}

\begin{claim}\label{claim:counting}
For any integer $0\le i<\frac{d-k-1}{2}$, let $\boldsymbol{v}\in \{0,1\}^{d}\setminus (V^{(0)}\cup \overline{V^{(0)}}\cup V^{(1)}\cup\overline{V^{(1)}}\cup\cdots\cup V^{(i)}\cup\overline{V^{(i)}})$, we have
    \begin{equation*}
        f(\boldsymbol{v})+f(\overline{\boldsymbol{v}})\le \frac{1}{2^{i+1}}+\frac{1}{2^{d-k-i-1}}.
    \end{equation*}
When $d-k$ is odd and $i=\frac{d-k-1}{2}$, we have
\begin{equation*}
        f(\boldsymbol{v})+f(\overline{\boldsymbol{v}})\le \frac{1}{2^{i}}.
    \end{equation*}

\end{claim}
\begin{poc}
We consider the following two cases.
\begin{enumerate}[(i)]
    \item If $\boldsymbol{v}$ and $\overline{\boldsymbol{v}}$ are covered by some two vectors in $U$, say $\boldsymbol{u}_{i}$ and $\boldsymbol{u}_{j}$, as $d_{H}(\boldsymbol{v},\overline{\boldsymbol{v}})=d$ and $d_{H}(\boldsymbol{u}_{i},\boldsymbol{u}_{j})\le k$, the sum of the number of jokers in $\boldsymbol{u}_{i}$ and $\boldsymbol{u}_{j}$ is at least $d-k$. We can see $f(\boldsymbol{v})=\frac{1}{2^{x}}$ and $f(\overline{\boldsymbol{v}})=\frac{1}{2^{y}}$ subject to $x\ge i+1$, $y\ge i+1$ and $x+y\ge d-k$. Then we have $f(\boldsymbol{v})+f(\overline{\boldsymbol{v}})\le \frac{1}{2^{i+1}}+\frac{1}{2^{d-k-i-1}}$ via simple optimization in the case $0 \leq i<\frac{d-k-1}{2}$. In the case $d-k$ is odd and $i=\frac{d-k-1}{2}$, we instead have $f(\boldsymbol{v})+f(\overline{\boldsymbol{v}})\le \frac{1}{2^{i+1}}+\frac{1}{2^{d-k-i}}=\frac{1}{2^{i}}$.
    \item If at least one of $\boldsymbol{v}$ and $\overline{\boldsymbol{v}}$ is not covered by any vector in $U$, then at least one of $f(\boldsymbol{v})$ and $f(\overline{\boldsymbol{v}})$ is equal to $0$. Without loss of generality, we can assume $\boldsymbol{v}$ is not be covered by any vector in $U$, then by Claim~\ref{claim:disjoint}, $\overline{\boldsymbol{v}}$ cannot be in $V^{(0)}\cup \overline{V^{(0)}}\cup V^{(1)}\cup\overline{V^{(1)}}\cup\cdots\cup V^{(i)}\cup\overline{V^{(i)}}$, which implies $f(\boldsymbol{v})+f(\overline{\boldsymbol{v}})\le\frac{1}{2^{i+1}}$. 
\end{enumerate}
\end{poc}

Next, we present the proof of Theorem~\ref{thm:main}.

\begin{proof}[Proof of Theorem~\ref{thm:main}]
We first consider the case that $d-k$ is even. Based on Claim~\ref{claim:disjoint}, Claim~\ref{claim:overline} Claim~\ref{claim:counting}, the identities $\sum_{\boldsymbol{v}\in\{0,1\}^{d}}f(\boldsymbol{v})=|U|$ and $|V^{(i)}|=|\overline{V^{(i)}}|$, the following upper bound holds if $0\le i\le \frac{d-k-2}{2}$.
\begin{equation}\label{equ:Counting}
   |U|=\sum_{\boldsymbol{v}\in\{0,1\}^{d}}f(\boldsymbol{v})\le \sum\limits_{j=0}^{i}\bigg(\frac{1}{2^{j}}+\frac{1}{2^{d-k-j}}\bigg)\cdot|V^{(j)}|+\bigg(\frac{1}{2^{i+1}}+\frac{1}{2^{d-k-i-1}}\bigg)\cdot \bigg(2^{d-1}-\sum_{j=0}^{i}|V^{(j)}|\bigg).
\end{equation}
Rearranging the terms, we obtain
\begin{equation}
    |U|\le \sum\limits_{j=0}^{i}\bigg(\frac{1}{2^{j}}+\frac{1}{2^{d-k-j}}-\frac{1}{2^{i+1}}-\frac{1}{2^{d-k-i-1}}\bigg)\cdot|V^{(j)}|+2^{d-i-2}+2^{k+i}.
\end{equation}

Write $S_t=\sum_{j=0}^t |V^{(j)}|$ for each $t \geq 0$, and set $S_{-1}=0$. We deduce that
\begin{align*}
    |U|&\le \sum\limits_{j=0}^{i}\bigg(\frac{1}{2^{j}}+\frac{1}{2^{d-k-j}}-\frac{1}{2^{i+1}}-\frac{1}{2^{d-k-i-1}}\bigg)\cdot|V^{(j)}|+2^{d-i-2}+2^{k+i}\\
    &= \sum\limits_{j=0}^{i}\bigg(\frac{1}{2^{j}}+\frac{1}{2^{d-k-j}}-\frac{1}{2^{i+1}}-\frac{1}{2^{d-k-i-1}}\bigg)\cdot(S_j-S_{j-1})+2^{d-i-2}+2^{k+i}\\
    &=\sum\limits_{j=0}^{i-1}\bigg(\frac{1}{2^{j}}+\frac{1}{2^{d-k-j}}-\frac{1}{2^{j+1}}-\frac{1}{2^{d-k-j-1}}\bigg) S_j+\bigg(\frac{1}{2^{i+1}}-\frac{1}{2^{d-k-i}}\bigg)S_i+2^{d-i-2}+2^{k+i}\\
    &=\sum\limits_{j=0}^{i}\bigg(\frac{1}{2^{j+1}}-\frac{1}{2^{d-k-j}}\bigg) S_j+2^{d-i-2}+2^{k+i}.
\end{align*} 
Combining Claim~\ref{claim:UpperBoundSum} and Theorem~\ref{thm:Kleitman}, we have
\begin{equation}\label{equ:CollectLikeTerms}
    |U|\le\sum\limits_{j=0}^{i}\bigg(\frac{1}{2^{j+1}}-\frac{1}{2^{d-k-j}}\bigg) \big|B_{k+2j}^{(d)}\big|+2^{d-i-2}+2^{k+i}.
\end{equation}

Next, we optimize the choice of $i$.
For each integer $0\le i\le \frac{d-k-2}{2}$, let
 \begin{equation*} g(i):=\sum\limits_{j=0}^{i}\bigg(\frac{1}{2^{j+1}}-\frac{1}{2^{d-k-j}}\bigg) \big|B_{k+2j}^{(d)}\big|+2^{d-i-2}+2^{k+i}.
\end{equation*}
It follows that for each $0\le i\le \frac{d-k-4}{2}$, we have
\begin{equation*}
    g(i+1)-g(i)= \bigg(\frac{1}{2^{d-k-i-1}}-\frac{1}{2^{i+2}}\bigg)\cdot\bigg(2^{d-1}-|B_{k+2(i+1)}^{(d)}|\bigg)
\end{equation*}
with $\frac{1}{2^{d-k-i-1}}<\frac{1}{2^{i+2}}$ and $|B_{k+2(i+1)}^{(d)}|<2^{d-1}$ in view of equation~\eqref{Bkd} and the binomial theorem. Thus, $g(i)$ is monotonically decreasing for integer $0\le i<\frac{d-k-1}{2}$. Hence, when $d-k$ is even, we can take $i=\frac{d-k-2}{2}$ in equation~\eqref{equ:CollectLikeTerms} and obtain the desired upper bound.

Finally, we discuss the case that $d-k$ is odd. In this case, $g(i)$ remains an upper bound for $|U|$ for each $0 \leq i \leq \frac{d-k-3}{2}$. Moreover, when $i=\frac{d-k-1}{2}$, Claim~\ref{claim:counting} indicates that 
$$
    |U| \leq \sum\limits_{j=0}^{i}\bigg(\frac{1}{2^{j}}+\frac{1}{2^{d-k-j}}-\frac{1}{2^i}\bigg)\cdot|V^{(j)}|+2^{\frac{d+k-1}{2}},
$$
and thus we have
\begin{equation}\label{equ:CollectLikeTerms2}
|U| \le \sum\limits_{j=0}^{i-1}\bigg(\frac{1}{2^{j+1}}-\frac{1}{2^{d-k-j}}\bigg) \big|B_{k+2j}^{(d)}\big|+\frac{1}{2^{d-k-i}}\big|B_{k+2i}^{(d)}\big|+2^{\frac{d+k-1}{2}},
\end{equation}
using a similar argument. Denote the above upper bound by $g(\frac{d-k-1}{2})$. In view of the discussion in the first case, it suffices to compare $g(\frac{d-k-3}{2})$ with $g(\frac{d-k-1}{2})$. In view of  equation~\eqref{Bkd} and the binomial theorem, it follows that
$$
g\bigg(\frac{d-k-1}{2}\bigg)-g\bigg(\frac{d-k-3}{2}\bigg)=\frac{1}{2^{\frac{d-k+1}{2}}}\big|B_{d-1}^{(d)}\big|+2^{\frac{d+k-1}{2}}-3\cdot 2^{\frac{d+k-3}{2}}=\frac{1}{2^{\frac{d-k+1}{2}}}\big|B_{d-1}^{(d)}\big|-2^{\frac{d+k-3}{2}}\leq 0.
$$
Therefore, $g(\frac{d-k-1}{2})$ gives the desired upper bound.
\end{proof}

Theorem~\ref{thm:main} also provides the exact value of $n(d-1,d)$ as follows, which has been shown in~\cite[Proposition 1]{2022neighbor}. 
\begin{cor}\label{cor:d-1}
For every $d\ge 2$, $n(d-1,d)=3\cdot 2^{d-2}$. 
\end{cor}

\begin{proof}
Theorem~\ref{thm:main} implies that $n(d-1,d)\leq 3\cdot 2^{d-2}$. On the other hand, by considering $U= \{11,10, 0*\} \square \{0,1\}^{d-2}$, we have $n(d-1,d)\geq 3\cdot 2^{d-2}$.
\end{proof}

\subsection{Comparison between Theorem~\ref{thm:main} and known results}\label{sec:comparison}

In this section, we compare Theorem~\ref{thm:main} with Theorem~\ref{thm:Huang2012} and Theorem~\ref{thm:previous}.

We first compare our new bound with Theorem~\ref{thm:previous}. 
Let $i$ be a positive integer such that $k \leqslant k+2 i-2 \leqslant d-1$. Note that 
$$
\sum_{j=0}^{\lceil(k+2i-2) / 2\rceil}\binom{d}{j}
\geq \big|B_{k+2i-2}^{(d)}\big|
>\sum\limits_{j=0}^{i-1}\bigg(\frac{1}{2^{j+1}}-\frac{1}{2^{d-k-j}}\bigg) \big|B_{k+2i-2}^{(d)}\big|
\geq \sum\limits_{j=0}^{i-1}\bigg(\frac{1}{2^{j+1}}-\frac{1}{2^{d-k-j}}\bigg) \big|B_{k+2j}^{(d)}\big|.
$$
It follows that
$$
\sum_{j=0}^{\lceil(k+2i-2) / 2\rceil}\binom{d}{j}+2^{d-i}>\sum\limits_{j=0}^{i-1}\bigg(\frac{1}{2^{j+1}}-\frac{1}{2^{d-k-j}}\bigg) \big|B_{k+2j}^{(d)}\big|+2^{d-i-1}+2^{k+i-1}=g(i-1)
$$
when $i\leq \frac{d-k}{2}$. When $i=\frac{d-k+1}{2}$, a similar argument shows that the left-hand side of the above inequality is strictly greater than $g(i)$. Thus, in view of the proof of Theorem~\ref{thm:main}, we see that Theorem~\ref{thm:main} always improves Theorem~\ref{thm:previous}.

\medskip

We also compare our upper bound with that in Theorem~\ref{thm:Huang2012}.
For ease, here we only consider the case that $d$ and $k$ are both sufficiently large even numbers, and the computations for the other cases are similar. Moreover, when $d\rightarrow\infty$ and $k>0.23d$, the upper bound in Theorem~\ref{thm:Huang2012} is worse than the trivial bound $2^d$. Thus it suffices to show that our upper bounds are better than that in Theorem~\ref{thm:Huang2012} when $d\rightarrow\infty$ and $0.123d\le k\le 0.23d$.

Theorem~\ref{thm:main} states that 
\begin{equation*}
     n(k,d) \leq \sum\limits_{j=0}^{\frac{d-k-2}{2}}\bigg(\frac{1}{2^{j+1}}-\frac{1}{2^{d-k-j}}\bigg) \big|B_{k+2j}^{(d)}\big|+2^{\frac{d+k}{2}}.
\end{equation*}
As $\big|B_{t}^{(d)}\big|=\sum\limits_{i=0}^{t/2}\binom{d}{i}$ for even number $t$, we have
\begin{equation*}
    n(k,d)\leq \sum_{t=0}^{\frac{k}{2}}\binom{d}{t}+\frac{1}{2}\binom{d}{\frac{k}{2}+1}+\frac{1}{4}\binom{d}{\frac{k}{2}+2}+\cdots +\frac{1}{2^{\frac{d-k-2}{2}}}\binom{d}{\frac{k}{2}+\frac{d-k-2}{2}}+2^{\frac{d+k}{2}}.
\end{equation*}
From the binomial theorem, it is easy to verify that $n(k,d)\leq 2^{\frac{k}{2}+1}\cdot \big(\frac{3}{2}\big)^d$.

We further consider the upper bound in Theorem~\ref{thm:Huang2012}. Consider the function 
$$h(x)=\frac{2^{x/2}}{x^x (1-x)^{1-x} \cdot \frac{3}{2}}.$$
One can see that $h(x)>1.01$ when $0.123 \leq x \leq 0.23$. Therefore, when $d \to \infty$ and $0.123d \leq k \leq 0.23d$, using Stirling approximation, we have
\begin{equation}\label{equ:C2}
    \frac{2^{k-1} \binom{d}{k}}{k \cdot 2^{\frac{k}{2}+1}\cdot (\frac{3}{2})^d} \sim \frac{2^{\frac{k}{2}-2}}{k} \sqrt{\frac{d}{2\pi k(d-k)}}\frac{d^d}{k^k (d-k)^{d-k}(\frac{3}{2})^d} \geq \frac{C}{d\sqrt{d}} \bigg(h\bigg(\frac{k}{d}\bigg)\bigg)^d> \frac{C}{d\sqrt{d}} 1.01^d,
\end{equation}
where $C$ is an absolute constant. 
We conclude that our bound improves the bound in Theorem~\ref{thm:Huang2012} exponentially when $d\rightarrow\infty$ and $0.123d \leq k \leq 0.23d$.

\subsection{Proof of Theorem~\ref{thm:main2}}\label{subsection:272829}

\begin{proof}[Proof of Theorem~\ref{thm:main2}]
Based on the proof of Theorem~\ref{thm:main}, we further consider two cases. If $V^{(0)}$ is empty, then the upper bound on $|U|$ follows from the proof of Theorem~\ref{thm:main} by ignoring the contribution of $V^{(0)}$.

Next, we assume that $V^{(0)}$ is not empty, that is, there is some vector $\boldsymbol{v}^{(0)}\in U \cap \{0,1\}^d$. Without loss of generality, we may assume that $\boldsymbol{v}^{(0)}=\boldsymbol{0}$. Let $H_{k}^{(d)}(\boldsymbol{0})$ be the $d$-dimensional Hamming ball in $\{0,1\}^d$ with radius $k$ centered at $\boldsymbol{0}$. Then for any vector $\boldsymbol{u}=(u_1, u_2, \ldots, u_d)\in U$, there is a vector $\boldsymbol{w}\in H_{k}^{(d)}(\boldsymbol{0})$ such that $\boldsymbol{u}$ covers $\boldsymbol{w}$. Indeed, we can take $\boldsymbol{w}=(w_1, w_2, \ldots, w_d) \in \{0,1\}^{d}$ with $w_i=u_i$ for $u_i \in \{0,1\}$ and $w_i=0$ for $u_i=*$, so that $d_{H}(\boldsymbol{0}, \boldsymbol{w})=d_{H}(\boldsymbol{0},\boldsymbol{u}) \leq k$. Therefore, by Claim~\ref{3.1}, we have
\begin{equation*}
    |U|\le |H_{k}^{(d)}(\boldsymbol{0})|=\sum_{j=0}^{k}\binom{d}{j}.
\end{equation*}
By taking the maximum of the upper bounds from the two cases, we get the desired upper bound.
\end{proof}

\begin{exam}\label{example:272727}
Using Theorem~\ref{thm:main}, we obtain that $n(2,7) \le 33$. Next, we show Theorem~\ref{thm:main2} implies a better upper bound $n(2,7) \le 29$. Suppose $V^{(0)}$ is not empty, then 
\begin{equation*}
    |U|\le |H_{2}^{(7)}(\boldsymbol{0})|=\sum_{j=0}^{2}\binom{7}{j}=29.
\end{equation*}
Therefore we can assume that $V^{(0)}=\emptyset$. By setting $i=1$ in inequality~\eqref{equ:CollectLikeTerms}, we have
\begin{equation*}
    |U|\le (\frac{1}{4}-\frac{1}{16})\cdot|V^{(1)}|+2^4+2^3=24+\frac{3}{16}\cdot|V^{(1)}|.
\end{equation*}
By Claim~\ref{claim:UpperBoundSum}, we have $|V^{(1)}| \leq 29$. Thus, $|U|\le 24+\frac{3}{16}\cdot 29=29.4375$. 
\end{exam}

\begin{rmk}\label{d_sufficiently_large}
For each fixed $k$, it is easy to see that Theorem~\ref{thm:main2} improves Theorem~\ref{thm:main} when $d$ is sufficiently large. In particular, when $k=2,3$, the optimal upper bound given in Theorem~\ref{thm:main2} is better than the optimal upper bound given in Theorem~\ref{thm:main} provided $d \geq 7$ and $d \geq 12$, respectively. For $k,d \leq 20$, We refer to Table~\ref{tab:values111} for an extensive list of the optimal upper bound on $n(k,d)$ given by Theorem~\ref{thm:main2} and Theorem~\ref{thm:main}.
\end{rmk}

We end the section by introducing a new parameter $h$ and giving a different upper bound on $n(k,d)$. This pushes all of our ideas in this paper to the limit.

\begin{theorem}\label{thm:main3}
    Let $1\leqslant k\leqslant d-1$ be integers. 
    If $d-k$ is even, then
\begin{equation}\label{eq11}
    n(k,d)\le \min\limits_{0\le h\le \frac{d-k-2}{2}} \max \left\{2^h\sum_{j=0}^k \binom{d-h}{j}, \sum\limits_{j=h+1}^{\frac{d-k-2}{2}}\bigg(\frac{1}{2^{j+1}}-\frac{1}{2^{d-k-j}}\bigg) \big|B_{k+2j}^{(d)}\big|+2^{\frac{d+k}{2}}\right\}.
\end{equation}
If $d-k$ is odd, then
  \begin{equation*}
    n(k,d)\le \min\limits_{0\le h\le \frac{d-k-3}{2}} \max \left\{2^h\sum_{j=0}^k \binom{d-h}{j}, \sum\limits_{j=h+1}^{\frac{d-k-3}{2}}\bigg(\frac{1}{2^{j+1}}-\frac{1}{2^{d-k-j}}\bigg) \big|B_{k+2j}^{(d)}\big|+\frac{1}{2^{\frac{d-k+1}{2}}}\big|B_{d-1}^{(d)}\big|+2^{\frac{d+k-1}{2}}\right\},
\end{equation*}
and also 
$$
  n(k,d)\le \max \left\{2^{\frac{d-k-1}{2}}\sum_{j=0}^k \binom{\frac{d+k+1}{2}}{j}, 2^{\frac{d+k-1}{2}}\right\}.
$$    
\end{theorem}

\begin{proof}
As in the proof of Theorem~\ref{thm:main2}, we can further consider whether $\bigcup_{j=0}^{h} V^{(j)}=\emptyset$ for $h>0$ and give a different upper bound on $n(k,d)$. We only prove the case that $d-k$ is even, since the analysis for the case that $d-k$ is odd is similar. Let $0\le h\le \frac{d-k-2}{2}$. 

If $\bigcup_{j=0}^{h} V^{(j)}=\emptyset$, then the upper bound on $|U|$ follows from the proof of Theorem~\ref{thm:main} by ignoring the contribution of $\bigcup_{j=0}^{h} V^{(j)}$. Thus, the second upper bound in inequality~\eqref{eq11} holds.

Next we assume $\bigcup_{j=0}^{h} V^{(j)}\neq \emptyset$. Without loss of generality we can assume that 
there is some vector $\boldsymbol{v}^{(0)}\in U$, where $\boldsymbol{v}^{(0)}=(v_1,v_2,\ldots, v_d)$ with $v_i=*$ for $i \leq t$ and $v_i=0$ for $i>t$, where $t \leq h$.  Then for any vector $\boldsymbol{u}=(u_1, u_2, \ldots, u_d)\in U$, there is a vector $\boldsymbol{w}\in\{0,1\}^h \square H_{k}^{(d-h)}(\boldsymbol{0})$ such that $\boldsymbol{u}$ covers $\boldsymbol{w}$. Indeed, since $d_{H}(\boldsymbol{v}^{(0)},\boldsymbol{u}) \leq k$, we can take $\boldsymbol{w}=(w_1, w_2, \ldots, w_d) \in\{0,1\}^h \square H_{k}^{(d-h)}(\boldsymbol{0})$ with $w_i=u_i$ for $u_i \in \{0,1\}$ and $w_i=0$ for $u_i=*$, so that $d_{H}(\boldsymbol{v}^{(0)}, \boldsymbol{w})=d_{H}(\boldsymbol{v}^{(0)},\boldsymbol{u}) \leq k$. Therefore, by Claim~\ref{3.1}, we have an injective map from $U$ to $\{0,1\}^h \square H_{k}^{(d-h)}(\boldsymbol{0})$ and thus $|U| \leq 2^h\sum_{j=0}^k \binom{d-h}{j}$, obtaining the first upper bound in inequality~\eqref{eq11}.  
\end{proof}
Note that the first bound of inequality~\eqref{eq11} is increasing in $h$, and the second bound of inequality~\eqref{eq11} is decreasing in $h$, so there is an optimal choice of $h$ when $k,d$ are fixed. In particular, by taking $h=1$, inequality~\eqref{eq11} implies that $n(2,10) \le 95$.

\section{Some exact values of $n(k,d)$: Proof of Theorem~\ref{thm:ShortTight}}\label{sec:values}

In this section, we prove Theorem~\ref{thm:ShortTight}. While it is tempting to use a unified approach to determine $n(5,7)$, $n(6,8)$, and $n(3,6)$, there are some subtle differences in the analysis of these three cases. A key observation we will use is that, for any pair of vectors $\boldsymbol{v}_{1},\boldsymbol{v}_{2}\in V^{(0)}$, we have $d_{H}(\boldsymbol{v}_{1},\boldsymbol{v}_{2})=d_{H}(\overline{\boldsymbol{v}_{1}},\overline{\boldsymbol{v}_{2}})$, which implies that, if $V^{(0)}$ is contained in some Hamming ball with radius $t$, so is $\overline{V^{(0)}}$. This observation, together with Theorem~\ref{thm:StabilityFrankl}, namely the stability version of Kleitman's Theorem, allows us to deduce that if $|U|$ is larger than some certain value, then $\overline{V^{(0)}}$ must be contained in some extremal configuration, which eventually leads to a contradiction.

We follow the notations in Section~\ref{sec:General} and let $U$ be an extremal configuration with $|U|=n(k,d)$.

\subsection{Exact value of $n(5,7)$}\label{sec:n(5,7)}
The goal of this section is to show $n(5,7)=74$. 
Recall that Theorem~\ref{thm:main} shows that $\sum_{\boldsymbol{v}\in\{0,1\}^{7}}f(\boldsymbol{v})\le 75$. 
By Corollary~\ref{cor:ToMain}, it suffices to show $n(5,7)\neq 75$. For the sake of contradiction, assume that $n(5,7)=75$. This forces all estimates in the proof of Theorem~\ref{thm:main} to be tight. In particular, in Claim~\ref{claim:UpperBoundSum} we must have $|V^{(0)}|=|\overline{V^{(0)}}|=44$. Moreover, by  Claim~\ref{claim:overline} and Claim~\ref{claim:counting}, $f(\boldsymbol{v})=1$ if $\boldsymbol{v}\in V^{(0)}$, $f(\boldsymbol{v})=\frac{1}{4}$ if $\boldsymbol{v}\in \overline{V^{(0)}}$ and $f(\boldsymbol{v})=\frac{1}{2}$ if $\boldsymbol{v}\in \{0,1\}^{7}\setminus(V^{(0)}\cup\overline{V^{(0)}})$. Therefore, there are exactly $11$ vectors in $U$ with $2$ jokers, and they exactly cover the $44$ vectors in $\overline{V^{(0)}}$.

On the other hand, note that $|V^{(0)}|=44$ is of maximum possible size. In view of Theorem~\ref{thm:StabilityFrankl} and the proof of Claim~\ref{claim:UpperBoundSum}, it follows that $V^{(0)}$ is precisely the Cartesian product of $\{0,1\}$ and a $6$-dimensional Hamming ball of radius $2$, and so is $\overline{V^{(0)}}$. Without loss of generality, assume that such a $6$-dimensional Hamming ball of radius $2$ is centered at the origin $\boldsymbol{0}\in\{0,1\}^{6}$, that is, assume $\overline{V^{(0)}}:=\{0,1\}\square \binom{[6]}{\le 2}$. Note that $\overline{V^{(0)}}$ are exactly covered by the $11$ vectors in $U$ with $2$ jokers.

\begin{claim}\label{claim:57575757}
    Any pair of vectors in $\{1\}\square\binom{[6]}{2}$ cannot be covered by the same vector in $\{0,1,*\}^{7}$ with $2$ jokers.
\end{claim}
\begin{poc}
    Suppose that there is a vector $\boldsymbol{u}\in\{0,1,*\}^{7}$ covering two different vectors $\boldsymbol{v},\boldsymbol{v}'\in\{1\}\square\binom{[6]}{2}$, then $\boldsymbol{u}$ has $2$ jokers and $\boldsymbol{u}$ covers $\boldsymbol{v}\lor\boldsymbol{v}'$. Note that the number of $1$-coordinates of $\boldsymbol{v}\lor\boldsymbol{v}'$ is at least $4$, thus $\boldsymbol{v}\lor\boldsymbol{v}'\notin \overline{V^{(0)}}$. However, recall that the vectors in $U$ with $2$ jokers do not cover any vector outside $\overline{V^{(0)}}$, a contradiction. 
\end{poc}
Therefore, each vector in $\{1\}\square\binom{[6]}{2}$ is covered by a unique vector in $U$ with $2$ jokers. However, $\binom{6}{2}=15>11$, a contradiction.

\subsection{Exact value of $n(6,8)$}
The goal of this section is to show $n(6,8)=150$. It suffices to show $n(6,8)\ne 151$ by Corollary \ref{cor:ToMain}. For the sake of contradiction, assume that $|U|=151$. Note that inequality~(\ref{equ:CollectLikeTerms}) in the proof of Theorem~\ref{thm:main} gives 
\begin{equation*}
    |U|\le 2^{6}+2^{6}+\bigg(\frac{1}{2}-\frac{1}{4}\bigg)\cdot|V^{(0)}|.
\end{equation*}
It follows that $|V^{(0)}|\geq 92$. Since Claim~\ref{claim:UpperBoundSum} implies that 
$|V^{(0)}|\le\sum_{j=0}^{3}\binom{8}{j}=93$, we must have $|V^{(0)}|=93$ or $|V^{(0)}|=92$. As in the proof in Section~\ref{sec:n(5,7)}, by Theorem~\ref{thm:StabilityFrankl}, $\overline{V^{(0)}}$ is contained in some $8$-dimensional Hamming balls of radius $3$, for otherwise $|V^{(0)}|\le 93-\binom{4}{3}+1=90$. Without loss of generality, assume that $\overline{V^{(0)}}\subseteq\binom{[8]}{\le 3}$. First we give the following analogous claim, which can be derived the same way as Claim~\ref{claim:57575757}, we omit the proof here.

\begin{claim}\label{claim:68686868}
If the vectors in $U$ with $2$ jokers do not cover any vector outside $\overline{V^{(0)}}$, then any pair of vectors in $\binom{[8]}{3}$ cannot be covered by same vector in $\{0,1,*\}^{8}$ with $2$ jokers.
\end{claim}

Then we consider the following cases.

\begin{Case}
\item Suppose that $|V^{(0)}|=92$. As in the proof in Section~\ref{sec:n(5,7)}, we must have $f(\boldsymbol{v})=1$ if $\boldsymbol{v}\in V^{(0)}$, $f(\boldsymbol{v})=\frac{1}{4}$ if $\boldsymbol{v}\in \overline{V^{(0)}}$ and $f(\boldsymbol{v})=\frac{1}{2}$ if $\boldsymbol{v}\in \{0,1\}^{8}\setminus(V^{(0)}\cup\overline{V^{(0)}})$. Therefore, there are exactly $23$ vectors in $U$ with $2$ jokers, and they exactly cover the $92$ vectors in $\overline{V^{(0)}}$. However, $\binom{8}{3}-1=55>23$, a contradiction to Claim~\ref{claim:68686868}.

\item Suppose that $|V^{(0)}|=93$. In this case, we have
$$
151=|U|=\sum_{\boldsymbol{v} \in \{0,1\}^8 } f(\boldsymbol{v})=93+\sum_{\boldsymbol{v}\in \overline{V^{(0)}}} f(\boldsymbol{v})+ \sum_{ \boldsymbol{v}\in \{0,1\}^8 \setminus (V^{(0)} \cup \overline{V^{(0)}})} f(\boldsymbol{v}) \leq 151+\frac{1}{4},
$$
where $f(\boldsymbol{v})\le \frac{1}{4}$ for each $\boldsymbol{v}\in \overline{V^{(0)}}$, and $f(\boldsymbol{v})\leq \frac{1}{2}$ for each $\boldsymbol{v}\in \{0,1\}^{8}\setminus(V^{(0)}\cup\overline{V^{(0)}})$. 
However, note that if $f(\boldsymbol{v})= \frac{1}{4}$ for all $\boldsymbol{v}\in \overline{V^{(0)}}$ and $f(\boldsymbol{v})= \frac{1}{2}$ for all $\boldsymbol{v}\in \{0,1\}^{8}\setminus(V^{(0)}\cup\overline{V^{(0)}})$, then the right-hand side of the above inequality holds and thus $|U|=151+\frac{1}{4}$, which is absurd. This observation leads to the following three different cases.

\begin{Case}
\item Suppose that there is a vector $\boldsymbol{v}\in \overline{V^{(0)}}$ such that $f(\boldsymbol{v})=0$. Then the above analysis shows that the remaining $92$ vectors in  $\overline{V^{(0)}}$ must have weight $\frac{1}{4}$, and $f(\boldsymbol{v})= \frac{1}{2}$ for all $\boldsymbol{v}\in \{0,1\}^{8}\setminus(V^{(0)}\cup\overline{V^{(0)}})$. Thus, 
there are exactly $23$ vectors in $U$ with $2$ jokers, and they exactly cover the $92$ vectors in $\overline{V^{(0)}}$ with nonzero weight. However, this is impossible by Claim~\ref{claim:68686868}.

\item Suppose that there is at least one vector in $\overline{V^{(0)}}$ satisfying $0<f(\boldsymbol{v})<\frac{1}{4}$, then there are at least $8$ vectors in $\{0,1\}^8$ that are covered by some vector in $U$ with at least $3$ jokers, which implies that
\begin{equation*}
    |U| \leq 151+\frac{1}{4}- 8 \bigg(\frac{1}{4}-\frac{1}{8}\bigg)<151.
\end{equation*}

\item Suppose that all of these $93$ vectors in  $\overline{V^{(0)}}$ are equipped with weight $\frac{1}{4}$. Since $4$ is not divided by $93$, there exist at least 3 vectors not in $\overline{V^{(0)}}$ with weight $\frac{1}{4}$. It follows that
\begin{equation*}
|U|\leq 151+\frac{1}{4}- 3\bigg(\frac{1}{2}-\frac{1}{4}\bigg)<151.
\end{equation*}  
\end{Case}
\end{Case}

\subsection{Exact value of $n(3,6)$}
The goal of this section is to show $n(3,6)=27$. It suffices to show $n(3,6)\neq 28$. For the sake of contradiction, assume that $|U|=28$. From inequality~(\ref{equ:CollectLikeTerms}) with $i=0$, we have 
\begin{equation*}
    |U|\le 2^{4}+2^{3}+\bigg(\frac{1}{2}-\frac{1}{8}\bigg)|V^{(0)}|.
\end{equation*}
It follows that $|V^{(0)}|\ge 11$. On the other hand, $|V^{(0)}|\le2\sum_{j=0}^{1}\binom{5}{j}=12$. Therefore, $|V^{(0)}|=12$ or $|V^{(0)}|=11$. By Theorem~\ref{thm:StabilityFrankl}, without loss of generality, we can further assume that $\overline{V^{(0)}} \subseteq \{0,1\}\square \binom{[5]}{\le 1}$.

Let $\delta=\frac{1}{2}$ if $|V^{(0)}|=12$, and  $\delta=\frac{1}{8}$ if $|V^{(0)}|=11$.
Note that we have
$$
|U|=\sum_{\boldsymbol{v} \in \{0,1\}^6 } f(\boldsymbol{v})=|V^{(0)}|+\sum_{\boldsymbol{v}\in \overline{V^{(0)}}} f(\boldsymbol{v})+ \sum_{ \boldsymbol{v}\in \{0,1\}^6 \setminus (V^{(0)} \cup\overline{V^{(0)}})} f(\boldsymbol{v}) \leq 28+\delta,
$$
where $f(\boldsymbol{v})\le \frac{1}{8}$ for $\boldsymbol{v}\in \overline{V^{(0)}}$, and $f(\boldsymbol{v})+f(\overline{\boldsymbol{v}})\leq \frac{1}{2}+\frac{1}{4}=\frac{3}{4}$ if $\boldsymbol{v}\in \{0,1\}^{6}\setminus(V^{(0)}\cup\overline{V^{(0)}})$. However, note that if $f(\boldsymbol{v})= \frac{1}{8}$ for all $\boldsymbol{v}\in \overline{V^{(0)}}$ and $f(\boldsymbol{v})+f(\overline{\boldsymbol{v}})= \frac{3}{4}$ for all $\boldsymbol{v}\in \{0,1\}^{6}\setminus(V^{(0)}\cup\overline{V^{(0)}})$, then the right-hand side of the above inequality holds and thus $|U|=28+\delta$, which is absurd. 

Suppose there is some vector $\boldsymbol{v}\in\overline{V^{(0)}}$ that is covered by a vector $\boldsymbol{u}\in U$ with at least $4$ jokers, then we have
$$
|U|\leq 28+\frac{1}{2}-16\bigg(\frac{1}{8}-\frac{1}{16}\bigg)<28,
$$
which is impossible. Thus, each vector $\boldsymbol{v}\in \overline{V^{(0)}}$ has either $f(\boldsymbol{v})= \frac{1}{8}$ or $f(\boldsymbol{v})=0$. Moreover, there are at most $4$ vectors $\boldsymbol{v}\in \overline{V^{(0)}}$ with $f(\boldsymbol{v})=0$, for otherwise
$$
|U|\leq 28+\frac{1}{2}-5 \cdot \frac{1}{8}<28.
$$
Therefore, $U$ has at least $1$ vector with $3$ jokers. 
\begin{Case}
\item Suppose that $U$ has only one vector $\boldsymbol{u}$ with $3$ jokers. Then $\boldsymbol{u}$ needs to cover at least $11-4=7$ vectors in  $\overline{V^{(0)}}$. However, it is easy to verify that $\boldsymbol{u}$ can cover at most $6$ vectors in $\{0,1\}\square \binom{[5]}{\le 1}$. 

\item Suppose that $U$ has at least $2$ vectors with $3$ jokers. 
\begin{Case}
    \item If $|V^{(0)}|=11$, then at least $5$ vectors in $\{0,1\}^{6}\setminus(V^{(0)}\cup\overline{V^{(0)}})$ have weight $\frac{1}{8}$ and thus
$$
|U|\leq 28+\frac{1}{8}-5 \cdot \bigg(\frac{1}{4}-\frac{1}{8}\bigg)<28.
$$
\item If $|V^{(0)}|=12$, then at least $4$ vectors in $\{0,1\}^{6}\setminus(V^{(0)}\cup\overline{V^{(0)}})$ have weight $\frac{1}{8}$ and thus
$$
|U|\leq 28+\frac{1}{2}-4 \cdot \bigg(\frac{1}{4}-\frac{1}{8}\bigg)=28.
$$
The assumption that $|U|=28$ then indicates that there are exactly $2$ vectors in $U$ with $3$ jokers, and they cover all $12$ vectors in $\overline{V^{(0)}}$. Note that any vector $\boldsymbol{u}\in \{0,1,*\}^6$  covers at most $6$ vectors in $\overline{V^{(0)}}=\{0,1\}\square\binom{[5]}{\le 1}$, and $\boldsymbol{u}$ covers $6$ vectors only if $\boldsymbol{u}\in \{0,1,*\}^6$ is of the form $\{*\}\square \boldsymbol{s}$, where $s$ contains $2$ jokers. However, it is easy to verify that any two vectors of such a form cannot cover all 12 vectors in $\{0,1\}\square\binom{[5]}{\le 1}$.
\end{Case}
\end{Case}

\section{Concluding remarks}\label{sec:Conclusion}
In this paper, we mainly focus on improving the upper bounds of the function $n(k,d)$ when $k$ is relatively close to $d$. The general improvement relies on a weighted count of all vectors in $\{0,1\}^{d}$. Furthermore, the authors in~\cite{2022neighbor} provided several better lower bounds on $n(k,d)$ via Gurobi (see Table~\ref{tab:values}) and they expected that all those solutions to the corresponding MIP problems should be optimal. Fortunately, we confirm that several of them are indeed optimal and the stability result of Kleitman's theorem is useful in the proofs. One can further improve the upper bounds of $n(k,d)$ via a similar argument in our proofs. However, in order to fully determine the values, even the asymptotic results of $n(k,d)$, new ideas are needed.

When $k>1$, the best-known general upper bound was obtained with tools of linear algebra~\cite{huang2012counterexample} until Alon, Grytczuk, Kisielewicz, and Przes\l awski~\cite{2022neighbor} improved the upper bound when $k$ is relatively close to $d$ without tools of linear algebra. Inspired by their results, we improve the upper bound in~\cite{huang2012counterexample} when $k=2$ and $d\le 10$ using purely combinatorial analysis, which yields that linear algebra methods do not provide tight bound on $n(k,d)$ even for $k=2$. Hence, improving the upper bound on $n(k,d)$ remains a challenging task in general.

\section*{Acknowledgement}
Zixiang Xu would like to thank Hong Liu and Jun Gao for introducing this problem, and thank Jarosław Grytczuk, Andrzej P. Kisielewicz, and Krzysztof Przes\l awski for sharing one of their constructions about $n(4,6)\ge 37$. The authors would like to express our gratitude to the anonymous reviewers for the detailed and constructive comments which are very helpful to the improvement of the technical presentation of this paper.

\medskip
\textbf{Note added} The first version of~\cite{2022neighbor} appeared on 9 Dec 2022. Our first draft was submitted to arXiv on 16 Jan 2023. Later, we were informed by the authors in~\cite{2022neighbor} that a new version (v3) of~\cite{2022neighbor} appeared in arXiv on 2 Feb 2023 and some previous bounds were further improved. Thus, we updated our abstract and listed their new results in our submitted version on March 5.

Recently, we found in the Appendix of~\cite{2022neighbor}, the authors stated that the exact values of $n(2,5)=12$, $n(2,6)=16$ and $n(2,7)=21$ were independently proved by Łuba~\cite{2023n(5.2)}.

\bibliographystyle{abbrv}
\bibliography{NeighborlyBox}

\begin{thebibliography}{10}

\bibitem{1997Alon}
N.~Alon.
\newblock Neighborly families of boxes and bipartite coverings.
\newblock In {\em The mathematics of {P}aul {E}rd\H{o}s, {II}}, volume~14 of {\em Algorithms Combin.}, pages 27--31. Springer, Berlin, 1997.

\bibitem{2022neighbor}
N.~Alon, J.~a. Grytczuk, A.~P. Kisielewicz, and K.~Przes\l~awski.
\newblock New bounds on the maximum number of neighborly boxes in {$\Bbb R^d$}.
\newblock {\em European J. Combin.}, 114:Paper No. 103797, 15, 2023.

\bibitem{2017CPCFrankl}
P.~Frankl.
\newblock A stability result for families with fixed diameter.
\newblock {\em Combin. Probab. Comput.}, 26(4):506--516, 2017.

\bibitem{2022Stability}
J.~Gao, H.~Liu, and Z.~Xu.
\newblock Stability through non-shadows.
\newblock {\em Combinatorica}, 43(6):1125--1137, 2023.

\bibitem{1972GrahamPollak}
R.~L. Graham and H.~O. Pollak.
\newblock On embedding graphs in squashed cubes.
\newblock In {\em Graph theory and applications ({P}roc. {C}onf., {W}estern {M}ichigan {U}niv., {K}alamazoo, {M}ich., 1972; dedicated to the memory of {J}. {W}. {T}. {Y}oungs)}, Lecture Notes in Math., Vol. 303, pages 99--110. Springer, Berlin, 1972.

\bibitem{huang2012counterexample}
H.~Huang and B.~Sudakov.
\newblock A counterexample to the {A}lon-{S}aks-{S}eymour conjecture and related problems.
\newblock {\em Combinatorica}, 32(2):205--219, 2012.

\bibitem{1966Kleitman}
D.~J. Kleitman.
\newblock On a combinatorial conjecture of {E}rd{\H{o}}s.
\newblock {\em J. Combinatorial Theory}, 1:209--214, 1966.

\bibitem{2023n(5.2)}
S.~{\L}uba.
\newblock Systems of unit cubes in {$\Bbb R^d$}.
\newblock {\em Uniwersytet Zielonogórski, 2023, (in Polish)}, Master thesis(https://github.com/saraluba/pracamagisterska), 2023.

\bibitem{1984Peck}
G.~W. Peck.
\newblock A new proof of a theorem of {G}raham and {P}ollak.
\newblock {\em Discrete Math.}, 49(3):327--328, 1984.

\bibitem{1982JGT}
H.~Tverberg.
\newblock On the decomposition of {$K_{n}$} into complete bipartite graphs.
\newblock {\em J. Graph Theory}, 6(4):493--494, 1982.

\bibitem{2008JCTA}
S.~Vishwanathan.
\newblock A polynomial space proof of the {G}raham-{P}ollak theorem.
\newblock {\em J. Combin. Theory Ser. A}, 115(4):674--676, 2008.

\bibitem{2013DMCountingProof}
S.~Vishwanathan.
\newblock A counting proof of the graham–pollak theorem.
\newblock {\em Discrete Mathematics}, 313(6):765--766, 2013.

\bibitem{1987Zaks}
J.~Zaks.
\newblock Neighborly families of congruent convex polytopes.
\newblock {\em Amer. Math. Monthly}, 94(2):151--155, 1987.

\end{thebibliography}

\appendix

\section{Appendix: Computations of improved upper bounds on $n(k,d)$}

In the following table, we list all $k,d \leq 20$ with $d-k \geq 2$ such that our new upper bound improves the best-known upper bound with the help of a simple C++ code. 

The third column records the best-known lower bound on $n(k,d)$, where we take the minimum among \cite[Theorem 2]{1997Alon}, \cite[Theorem 2]{2022neighbor}, \cite[Proposition 2]{2022neighbor}, and the computation by Gurobi in \cite[Table 2]{2022neighbor}). Note that we take the recursive lower bound $n(k_1+k_2,d_1+d_2) \geq n(k_1,d_1) n(k_2,d_2)$ proved in \cite[Proposition 2]{2022neighbor} into account, and our computations show that repeating using such a recursion leads to the best lower bound for most $k$ and $d$ in our table.

The fourth column lists the best-known upper bound (taking the maximum from \cite{huang2012counterexample} and \cite[Theorem 1]{2022neighbor}), and the last column lists the best upper bound from the results (Theorem~\ref{thm:main}, Theorem~\ref{thm:main2}, and Theorem~\ref{thm:main3}) in this paper. In particular, if the upper bound given in Theorem~\ref{thm:main3} is the smallest, we mark $*$ in the last column.

\begin{table}[ht]
    \centering
\begin{tabular}{|c|c|c|c|c|}
	\hline
$k$	& $d$ & Best-known lower bound & Best-known upper bound & Our new upper bound  \\
	\hline
2 & 4 & 9 & 13 & 9 \\   \hline
2 & 5 & 12 & 22 & 14 \\         \hline
2 & 6 & 16 & 37 & 21 \\         \hline
2 & 7 & 21 & 50 & 29 \\         \hline
2 & 8 & 27 & 65 & 45 \\         \hline
2 & 9 & 30 & 82 & 70 \\         \hline
2 & 10 & 36 & 101 & 95* \\         \hline
3 & 5 & 18 & 32 & 18 \\         \hline
3 & 6 & 27 & 54 & 27 \\         \hline
3 & 7 & 37 & 93 & 43 \\         \hline
3 & 8 & 48 & 157 & 66 \\        \hline
3 & 9 & 64 & 258 & 100 \\       \hline
3 & 10 & 84 & 432 & 151 \\      \hline
3 & 11 & 108 & 744 & 228 \\     \hline
3 & 12 & 135 & 1025 & 332 \\    \hline
3 & 13 & 162 & 1314 & 504 \\    \hline
3 & 14 & 189 & 1653 & 756* \\    \hline
3 & 15 & 216 & 2046 & 1097* \\   \hline
3 & 16 & 252 & 2497 & 1673* \\   \hline
3 & 17 & 294 & 3010 & 2365* \\   \hline
4 & 6 & 37 & 54 & 37 \\         \hline
4 & 7 & 54 & 93 & 58 \\         \hline
4 & 8 & 81 & 157 & 91 \\        \hline
4 & 9 & 111 & 258 & 141 \\      \hline
4 & 10 & 148 & 432 & 217 \\     \hline
4 & 11 & 192 & 744 & 332 \\     \hline
4 & 12 & 256 & 1306 & 504 \\    \hline
4 & 13 & 336 & 2117 & 762 \\    \hline
4 & 14 & 441 & 3519 & 1150 \\   \hline
4 & 15 & 567 & 6037 & 1733 \\   \hline
4 & 16 & 729 & 10709 & 2539 \\  \hline
4 & 17 & 810 & 17594 & 3844 \\  \hline
4 & 18 & 972 & 28069 & 5804 \\  \hline
4 & 19 & 1134 & 35246 & 8459* \\         \hline
4 & 20 & 1323 & 43721 & 12834* \\        \hline
5 & 7 & 74 & 128 & 74 \\        \hline
5 & 8 & 114 & 221 & 117 \\      \hline
5 & 9 & 162 & 384 & 183 \\      \hline
5 & 10 & 243 & 642 & 283 \\     \hline
5 & 11 & 333 & 1074 & 435 \\    \hline
5 & 12 & 444 & 1818 & 664 \\    \hline
5 & 13 & 592 & 3141 & 1008 \\   \hline
5 & 14 & 777 & 5521 & 1524 \\   \hline
5 & 15 & 1024 & 9040 & 2301 \\  \hline
5 & 16 & 1344 & 15077 & 3467 \\         \hline
5 & 17 & 1764 & 25786 & 5216 \\         \hline
5 & 18 & 2268 & 45384 & 7842 \\         \hline
5 & 19 & 2916 & 76564 & 11781 \\        \hline
\end{tabular}
\end{table}
\begin{table}[ht]
    \centering
\begin{tabular}{|c|c|c|c|c|}
	\hline
$k$	& $d$ & Best-known lower bound & Best-known upper bound & Our new upper bound  \\
	\hline
5 & 20 & 3645 & 125996 & 17690 \\       \hline
6 & 8 & 150 & 221 & 150 \\      \hline
6 & 9 & 228 & 384 & 240 \\      \hline
6 & 10 & 342 & 642 & 381 \\     \hline
6 & 11 & 486 & 1074 & 598 \\    \hline
6 & 12 & 729 & 1818 & 929 \\    \hline
6 & 13 & 999 & 3141 & 1433 \\   \hline
6 & 14 & 1369 & 5521 & 2195 \\  \hline
6 & 15 & 1776 & 9040 & 3346 \\  \hline
6 & 16 & 2368 & 15077 & 5079 \\         \hline
6 & 17 & 3108 & 25786 & 7688 \\         \hline
6 & 18 & 4096 & 45384 & 11609 \\        \hline
6 & 19 & 5376 & 76564 & 17499 \\        \hline
6 & 20 & 7056 & 125996 & 26345 \\       \hline
7 & 9 & 300 & 512 & 302 \\      \hline
7 & 10 & 456 & 894 & 481 \\     \hline
7 & 11 & 684 & 1536 & 762 \\    \hline
7 & 12 & 1026 & 2610 & 1196 \\   \hline
7 & 13 & 1458 & 4428 & 1859 \\  \hline
7 & 14 & 2187 & 7569 & 2866 \\  \hline
7 & 15 & 2997 & 13136 & 4391 \\         \hline
7 & 16 & 4107 & 23085 & 6692 \\         \hline
7 & 17 & 5476 & 38162 & 10159 \\        \hline
7 & 18 & 7104 & 63948 & 15376 \\        \hline
7 & 19 & 9472 & 109332 & 23218 \\       \hline
7 & 20 & 12432 & 191532 & 34999 \\      \hline
8 & 10 & 600 & 898 & 608 \\     \hline
8 & 11 & 912 & 1536 & 978 \\    \hline
8 & 12 & 1369 & 2610 & 1569 \\  \hline
8 & 13 & 2052 & 4428 & 2494 \\  \hline
8 & 14 & 3078 & 7569 & 3924 \\  \hline
8 & 15 & 4374 & 13136 & 6118 \\         \hline
8 & 16 & 6561 & 23085 & 9463 \\         \hline
8 & 17 & 8991 & 38162 & 14543 \\        \hline
8 & 18 & 12321 & 63948 & 22230 \\       \hline
8 & 19 & 16428 & 109332 & 33839 \\      \hline
8 & 20 & 21904 & 191532 & 51339 \\      \hline
9 & 11 & 1200 & 2048 & 1217 \\  \hline
9 & 12 & 1824 & 3534 & 1957 \\  \hline
9 & 13 & 2738 & 6144 & 3139 \\  \hline
9 & 14 & 4218 & 10572 & 4989 \\         \hline
9 & 15 & 6156 & 18141 & 7849 \\         \hline
9 & 16 & 9234 & 31277 & 12237 \\        \hline
9 & 17 & 13122 & 54546 & 18927 \\       \hline
9 & 18 & 19683 & 95772 & 29086 \\       \hline
9 & 19 & 26973 & 159720 & 44461 \\      \hline
9 & 20 & 36963 & 269052 & 67679 \\      \hline
10 & 12 & 2400 & 3634 & 2444 \\         \hline
\end{tabular}
\end{table}
\begin{table}[ht]
    \centering
\begin{tabular}{|c|c|c|c|c|}
	\hline
$k$	& $d$ & Best-known lower bound & Best-known upper bound & Our new upper bound  \\
	\hline
10 & 13 & 3648 & 6144 & 3964 \\         \hline
10 & 14 & 5550 & 10572 & 6424 \\        \hline
10 & 15 & 8436 & 18141 & 10326 \\       \hline
10 & 16 & 12996 & 31277 & 16430 \\      \hline
10 & 17 & 18468 & 54546 & 25884 \\      \hline
10 & 18 & 27702 & 95772 & 40421 \\      \hline
10 & 19 & 39366 & 159720 & 62648 \\     \hline
10 & 20 & 59049 & 269052 & 96485 \\     \hline
11 & 13 & 4800 & 8192 & 4889 \\         \hline
11 & 14 & 7296 & 14004 & 7929 \\        \hline
11 & 15 & 11100 & 24576 & 12849 \\      \hline
11 & 16 & 17100 & 42717 & 20653 \\      \hline
11 & 17 & 25992 & 73994 & 32861 \\      \hline
11 & 18 & 38988 & 128540 & 51769 \\     \hline
11 & 19 & 55404 & 225256 & 80843 \\     \hline
11 & 20 & 83106 & 395022 & 125296 \\    \hline
12 & 14 & 9600 & 14668 & 9811 \\        \hline
12 & 15 & 14592 & 24576 & 16018 \\      \hline
12 & 16 & 22500 & 42717 & 26191 \\      \hline
12 & 17 & 34200 & 73994 & 42514 \\      \hline
12 & 18 & 51984 & 128540 & 68325 \\     \hline
12 & 19 & 77976 & 225256 & 108697 \\    \hline
12 & 20 & 116964 & 395022 & 171313 \\   \hline
13 & 15 & 19200 & 32768 & 19622 \\      \hline
13 & 16 & 29184 & 55587 & 32037 \\      \hline
13 & 17 & 45000 & 98304 & 52382 \\      \hline
13 & 18 & 68400 & 172298 & 85028 \\     \hline
13 & 19 & 103968 & 300838 & 136651 \\    \hline
13 & 20 & 156066 & 526094 & 217395 \\   \hline
14 & 16 & 38400 & 59101 & 39351 \\      \hline
14 & 17 & 58368 & 98304 & 64611 \\      \hline
14 & 18 & 90000 & 172298 & 106445 \\    \hline
14 & 19 & 136800 & 300838 & 174283 \\   \hline
14 & 20 & 207936 & 526094 & 282639 \\   \hline
15 & 17 & 76800 & 131072 & 78702 \\     \hline
15 & 18 & 116736 & 220918 & 129223 \\   \hline
15 & 19 & 180000 & 393216 & 212891 \\   \hline
15 & 20 & 273600 & 694054 & 348567 \\   \hline
16 & 18 & 153600 & 237834 & 157762 \\   \hline
16 & 19 & 233472 & 393216 & 260270 \\   \hline
16 & 20 & 360000 & 694054 & 431610 \\   \hline
17 & 19 & 307200 & 524288 & 315525 \\   \hline
17 & 20 & 466944 & 878810 & 520540 \\   \hline
18 & 20 & 614400 & 956198 & 632265 \\   \hline
\end{tabular}
\caption{Improved upper bounds on $n(k,d)$ for $k,d \leq 20$ }
    \label{tab:values111}

\end{table}

\end{document}